\documentclass[11pt]{amsart}
\usepackage{mathrsfs}
\usepackage{amssymb}

\usepackage{amscd}
\usepackage{amsmath, amssymb}
\usepackage{amsfonts}
\usepackage[colorlinks,linkcolor=blue,citecolor=blue, pdfstartview=FitH]{hyperref}
\usepackage[all]{xy}

\numberwithin{equation}{section}

\theoremstyle{plain}
\newtheorem{thm}{Theorem}[section]
\newtheorem{theorem}[thm]{Theorem}
\newtheorem{lemma}[thm]{Lemma}
\newtheorem{corollary}[thm]{Corollary}
\newtheorem{proposition}[thm]{Proposition}

\theoremstyle{definition}

\newtheorem{defn-thm}[thm]{Definition-Theorem}

\newcommand{\Image}{{ \textrm{Im}~}}
\newcommand{\Hom}{{ \textrm{Hom}~}}

\begin{document}
\title{On the obstructions of deforming vector forms}
\author{Wei Xia}
\address{Wei Xia, School of Mathematics, Sun Yat-sen University, Guangzhou, P.R.China, 510275.} \email{xiaweiwei3@126.com, xiaw9@mail.sysu.edu.cn}

\thanks{This work was supported by the National Natural Science Foundations of China No. 11901590.}

\begin{abstract}
In this note, we show that the obstruction classes of deforming vector forms on a compact K\"ahler manifold is annihilated by cohomology classes.

\vskip10pt
\noindent
{\bf MSC~Classification (2010):} 32G05, 32L10, 55N30, 32G99
\end{abstract}

\maketitle
%
%
\section
{\bf Introduction}
Let $X$ be a compact K\"ahler manifold, then a well-known theorem of Clemens \cite[Thm.\,10.1]{Cle99} says that the obstruction classes of deforming the complex structure on $X$ is annihilated by cohomology classes:
\[
\{\text{obstructions}\}\subseteq \ker\left\{ H^{2}(X,T^{1,0})\longrightarrow \bigoplus_{r,s}\Hom (H^{r,s}(X), H^{r-1,s+2}(X)) \right\}.
\]
See \cite{Man04} for an algebraic proof of this result from the DGLA point of view. The main objective of this note is to prove an analogue of this theorem in the case of deforming vector forms on K\"ahler manifolds. In fact, denote by $\mathbb{C}\{t\}$ be the ring of convergent power series in $t\in \mathbb{C}^N$, $\mathfrak{m}$ the maximal ideal of $\mathbb{C}\{t\}$ and $\mathfrak{q}\subseteq\mathfrak{p}\subseteq\mathfrak{m}$ be proper ideals, we have
\begin{theorem}\label{main result 0}
Let $\phi=\phi(t)\in A^{0,1}(X,T^{1,0})\otimes \mathfrak{m}$ be a deformation of the compact K\"ahler manifold $X$ over $Spec \frac{\mathbb{C}\{t\}}{\mathfrak{q}}$. Assume $\sigma \in A^{0,q}(X,T^{1,0})\otimes \mathbb{C}\{t\}$ is a deformation of $[y]\in H_{\bar{\partial}}^{0,q}(X,T^{1,0})$ on $Spec \frac{\mathbb{C}\{t\}}{\mathfrak{p}}$ with $\mathfrak{q}\subseteq\mathfrak{p}\subseteq\mathfrak{m}$, then we have
\[
i_{[\phi,\sigma]}\omega_0=0\in H_{\bar{\partial}}^{r-1,s+q+1}(X)\otimes \frac{\mathfrak{p}}{\mathfrak{mp}},
\]
for any $\omega_0\in \mathcal{H}^{r,s}(X)$.
\end{theorem}
This result can also be intuitively reformulated as follows:
\[
\{\text{obstructions}\}\subseteq \ker\left\{ H^{q+1}(X,T^{1,0})\longrightarrow \bigoplus_{r,s}\Hom (H^{r,s}(X), H^{r-1,s+q+1}(X)) \right\}.
\]
We note that the K\"ahler assumption is not necessary here. In order that Theorem \ref{main result 0} to hold, it is enough to assume the deformations of $(p,q)$-forms on $X$ are unobstructed. This will be clear from the proof given in Section \ref{sec-Main-result}.

\section{Background and terminology}
\label{sec-Background-terminology}
In this section, we review some background material and set up our terminology that will be used afterwards.

By definition~\cite{Cat13}, a \emph{deformation} of a compact complex space $X$ is a flat proper morphism $\pi: (\mathcal{X}, X_{0}) \to (B, 0)$ between connected complex spaces with $X_{0}:=\pi^{-1}(0)\cong X$ and a germ of deformation is called a \emph{small deformation}. In this paper, we will only consider small deformations of compact complex manifolds and simply call these deformations.

Denote by $\mathbb{C}\{t\}$ be the ring of convergent power series in $t\in \mathbb{C}^N$ and $\mathfrak{m}$ the maximal ideal of $\mathbb{C}\{t\}$. For any proper ideal $\mathfrak{q}\subset\mathfrak{m}$, the analytic spectra $Spec \frac{\mathbb{C}\{t\}}{\mathfrak{q}}$ of the analytic algebra $\frac{\mathbb{C}\{t\}}{\mathfrak{q}}$ is the analytic singularity $(\mathbb{C}^N, 0, \mathfrak{q})$~\cite[pp.\,36]{Ma05}. Equivalently, The (small) deformation $\pi: (\mathcal{X}, X_{0}) \to (B, 0)$ may be given by $\phi\in A^{0,1}(X,T^{1,0})\otimes \mathfrak{m}$ such that $\bar{\partial}\phi - \frac{1}{2}[\phi ,\phi ]\in A^{0,2}(X,T^{1,0})\otimes \mathfrak{q}$ and $(B, 0)=Spec \frac{\mathbb{C}\{t\}}{\mathfrak{q}}$ as analytic singularities. Let $\phi\in A^{0,1}(X,T^{1,0})\otimes \mathfrak{m}$ be a deformation of $X$ over $Spec \frac{\mathbb{C}\{t\}}{\mathfrak{q}}$ and $\mathfrak{q}'$ be a proper ideal of $\mathbb{C}\{t\}$ such that $\mathfrak{mq}\subseteq\mathfrak{q}'\subseteq\mathfrak{q}\subseteq\mathfrak{m}$, we say $\phi$ has an \emph{extension} from $Spec \frac{\mathbb{C}\{t\}}{\mathfrak{q}}$ to $Spec \frac{\mathbb{C}\{t\}}{\mathfrak{q}'}$ if
there exists $\alpha\in A^{0,1}(X,T^{1,0})\otimes \mathfrak{q}$ such that $\phi-\alpha$ is a deformation over $Spec \frac{\mathbb{C}\{t\}}{\mathfrak{q}'}$. $[\phi,\phi]\in H^{2}(X,T^{1,0})\otimes \frac{\mathfrak{q}}{\mathfrak{mq}}$ is called the \emph{obstruction class}.

Denote the space of harmonic $(p,q)$-forms by $\mathcal{H}^{p,q}(X)$. Let $\phi\in A^{0,1}(X,T^{1,0})\otimes \mathfrak{m}$ be a deformation of the compact K\"ahler manifold $X$ over $Spec \frac{\mathbb{C}\{t\}}{\mathfrak{q}}$. The theorem of Clemens \cite[Thm.\,10.1]{Cle99} says that
\[
i_{[\phi,\phi]}\omega_0=0\in H_{\bar{\partial}}^{r-1,s+2}(X)\otimes \frac{\mathfrak{q}}{\mathfrak{mq}},
\]
for any $\omega_0\in \mathcal{H}^{r,s}(X)$.

In our previous work \cite{Xia19dDol}, a deformation theory is established for Dolbeault cohomology classes valued in holomorphic tensor bundles. In this paper, we will concentrate on the case of tangent bundle, i.e. deformations of vector forms. For our purpose, we may reformulate the definition for deformations of vector forms as follows. Let $\phi\in A^{0,1}(X,T^{1,0})\otimes \mathfrak{m}$ be a deformation of $X$ over $Spec \frac{\mathbb{C}\{t\}}{\mathfrak{q}}$ and $\mathfrak{p}$ be a proper ideal of $\mathbb{C}\{t\}$ such that $\mathfrak{q}\subseteq\mathfrak{p}\subseteq\mathfrak{m}$. For any $[y]\in H_{\bar{\partial}}^{0,q}(X,T^{1,0})$, a \emph{deformation} of $[y]$ (w.r.t. $\phi$ ) on $Spec \frac{\mathbb{C}\{t\}}{\mathfrak{p}}$ is given by $\sigma \in A^{0,q}(X,T^{1,0})\otimes \mathbb{C}\{t\}$ such that
\begin{itemize}
  \item[1.] $\sigma -y=0\in H_{\bar{\partial}}^{0,q}(X,E)\otimes \frac{\mathbb{C}\{t\}}{\mathfrak{m}}$;
  \item[2.] $\bar{\partial}_{\phi}\sigma  := \bar{\partial}\sigma  - [\phi  , \sigma  ] \in A^{0,q}(X,E)\otimes \mathfrak{p}$.
\end{itemize}
Likewise, let $\sigma \in A^{0,q}(X,T^{1,0})\otimes \mathfrak{m}$ be a deformation of $[y]\in H_{\bar{\partial}}^{0,q}(X,T^{1,0})$ on $Spec \frac{\mathbb{C}\{t\}}{\mathfrak{p}}$ and $\mathfrak{p}'$ be a proper ideal of $\mathbb{C}\{t\}$ such that $\mathfrak{q}\subseteq\mathfrak{mp}\subseteq\mathfrak{p}'\subseteq\mathfrak{p}\subseteq\mathfrak{m}$, we say $\sigma$ has an \emph{extension} from $Spec \frac{\mathbb{C}\{t\}}{\mathfrak{p}}$ to $Spec \frac{\mathbb{C}\{t\}}{\mathfrak{p}'}$ if
there exists $\beta\in A^{0,q}(X,T^{1,0})\otimes \mathfrak{q}$ such that $\sigma-\beta$ is a deformation over $Spec \frac{\mathbb{C}\{t\}}{\mathfrak{p}'}$. $[\phi,\sigma]\in H^{q+1}(X,T^{1,0})\otimes \frac{\mathfrak{p}}{\mathfrak{mp}}$ is called the \emph{obstruction class}. Indeed, we have modulo $\mathfrak{mp}$,
\begin{align*}
\bar{\partial}[\phi,\sigma]&=[\bar{\partial}\phi,\sigma]-[\phi,\bar{\partial}\sigma]\\
&\equiv \frac{1}{2}[[\phi ,\phi ],\sigma] - [\phi,[\phi,\sigma]]\\
&=\frac{1}{2}[[\phi ,\phi ],\sigma] - \frac{1}{2}[[\phi ,\phi ],\sigma]\\
&=0 .
\end{align*}
On the other hand, if there exists $\beta\in A^{0,q}(X,T^{1,0})\otimes \mathfrak{q}$ such that $\bar{\partial}\sigma  - [\phi  , \sigma  ]- \bar{\partial}\beta\in \mathfrak{p}'$, then $\sigma-\beta$ is a extension of $\sigma$ from $Spec \frac{\mathbb{C}\{t\}}{\mathfrak{p}}$ to $Spec \frac{\mathbb{C}\{t\}}{\mathfrak{p}'}$.

The following result will be essential for us:
\begin{proposition}\label{prop-canonical-deformation}
Let $X$ be a compact K\"ahler manifold and $\phi=\phi(t)\in A^{0,1}(X,T^{1,0})\otimes \mathfrak{m}$ be a deformation of the compact complex manifold $X$ over $Spec \frac{\mathbb{C}\{t\}}{\mathfrak{q}}$. Then given any $\omega_0\in \mathcal{H}^{r,s}(X)$, there exists a deformation $\omega\in A^{p,q}(X)\otimes \mathbb{C}\{t\}$ of $\omega_0$ on $Spec \frac{\mathbb{C}\{t\}}{\mathfrak{q}}$.
\end{proposition}
\begin{proof}This follows from the fact that the deformations of $(p,q)$-forms on K\"ahler manifolds are unobstructed, see \cite[Thm.\,8.1]{Xia19dDol} or \cite[Thm.\,10.1]{Cle99}.
\end{proof}
Furthermore, we recall the following useful formula from \cite[Lem.\,3.2]{LRY15}:
\begin{equation}\label{eq-Li}
[\mathcal{L}^{1,0}_{\phi},i_{\psi}]=\mathcal{L}^{1,0}_{\phi}i_{\psi}-(-1)^{k(l-1)}i_{\psi}\mathcal{L}^{1,0}_{\phi}=i_{[\phi,\psi]},
\end{equation}
for any $\phi\in A^{0,k}(X,T^{1,0})$, $\psi\in A^{0,l}(X,T^{1,0})$. See also \cite[Lem.\,3.7]{Xia18da}.

\section{Main result}\label{sec-Main-result}
\begin{lemma} \label{lem}
Let $\phi=\phi(t)\in A^{0,1}(X,T^{1,0})\otimes \mathfrak{m}$ be a deformation of any compact complex manifold $X$ over $Spec \frac{\mathbb{C}\{t\}}{\mathfrak{q}}$. Then
\[
\mathcal{H}^{p,q}(X)\cap\Image\bar{\partial}_{\phi(t)}=0.
\]
\end{lemma}
\begin{proof}See the proof of \cite[Thm.\,8.1]{Xia19dDol}. In fact, let $t\in Spec \frac{\mathbb{C}\{t\}}{\mathfrak{q}}$ be fixed and consider the family of vector spaces $\mathcal{H}_{\phi(s)}^{p,q}(X)\cap\Image\bar{\partial}_{\phi(t)}$. We may apply the upper semi-continuity theorem of Kodaira-Spencer to the analytic family of operators $\Box_{\phi(s)}$.
For any $s\in Spec \frac{\mathbb{C}\{t\}}{\mathfrak{q}}$ sufficiently near to $t$ we have
\[
\dim \mathcal{H}_{\phi(s)}^{p,q}(X)\cap\Image\bar{\partial}_{\phi(t)}\leq \dim \mathcal{H}_{\phi(t)}^{p,q}(X)\cap\Image\bar{\partial}_{\phi(t)}=0.
\]
The conclusion then follows this inequality by setting $s=0$ and the fact that $\mathcal{H}^{p,q}(X)\cap\Image\bar{\partial}_{\phi(t)}=\mathcal{H}_{\phi(0)}^{p,q}(X)\cap\Image\bar{\partial}_{\phi(t)}$.
\end{proof}

\begin{theorem}\label{main result}
Let $\phi=\phi(t)\in A^{0,1}(X,T^{1,0})\otimes \mathfrak{m}$ be a deformation of the compact K\"ahler manifold $X$ over $Spec \frac{\mathbb{C}\{t\}}{\mathfrak{q}}$. Assume $\sigma \in A^{0,q}(X,T^{1,0})\otimes \mathbb{C}\{t\}$ is a deformation of $[y]\in H_{\bar{\partial}}^{0,q}(X,T^{1,0})$ on $Spec \frac{\mathbb{C}\{t\}}{\mathfrak{p}}$ with $\mathfrak{q}\subseteq\mathfrak{p}\subseteq\mathfrak{m}$, then we have
\[
i_{[\phi,\sigma]}\omega_0=0\in H_{\bar{\partial}}^{r-1,s+q+1}(X)\otimes \frac{\mathfrak{p}}{\mathfrak{mp}},
\]
for any $\omega_0\in \mathcal{H}^{r,s}(X)$.
\end{theorem}
\begin{proof}Given any $\omega_0\in \mathcal{H}^{r,s}(X)$, let $\omega$ be a deformation of $\omega_0$ on $Spec \frac{\mathbb{C}\{t\}}{\mathfrak{q}}$ as assured by Proposition \ref{prop-canonical-deformation}. We compute modulo $\mathfrak{mp}$,
\begin{align*}
\bar{\partial}i_{\sigma}\omega&=i_{\bar{\partial}\sigma}\omega+(-1)^{q-1}i_{\sigma}\bar{\partial}\omega\\
(\bar{\partial}\omega-\mathcal{L}_{\phi}^{1,0}\omega\in \mathfrak{q}\subseteq\mathfrak{mp})
&\equiv i_{\bar{\partial}\sigma}\omega+(-1)^{q-1}i_{\sigma}\mathcal{L}_{\phi}^{1,0}\omega\\
(~\text{use}~~\eqref{eq-Li}~)
&= i_{\bar{\partial}\sigma}\omega - i_{[\phi,\sigma]}\omega + \mathcal{L}_{\phi}^{1,0}i_{\sigma}\omega\\
&= i_{\bar{\partial}\sigma-[\phi,\sigma]}\omega + \mathcal{L}_{\phi}^{1,0}i_{\sigma}\omega\\
(\bar{\partial}\sigma - [\phi, \sigma]\in \mathfrak{p}~\text{and}~\omega-\omega_0\in\mathfrak{m})
&\equiv i_{\bar{\partial}\sigma-[\phi,\sigma]}\omega_0 + \mathcal{L}_{\phi}^{1,0}i_{\sigma}\omega\\
&= \bar{\partial}i_{\sigma}\omega_0-i_{[\phi,\sigma]}\omega_0 + \mathcal{L}_{\phi}^{1,0}i_{\sigma}\omega,
\end{align*}
which implies
\[
i_{[\phi,\sigma]}\omega_0 + \bar{\partial}_{\phi}i_{\sigma}\omega \equiv \bar{\partial}i_{\sigma}\omega_0.
\]
The conclusion then follows from Lemma \ref{lem}.
\end{proof}

We remark that Theorem \ref{main result} still holds if we only assume the deformations of $(p,q)$-forms on $X$ (possibly non-K\"ahler) are unobstructed. Furthermore, if the canonical bundle $K_X$ of $X$ is trivial, then
\begin{equation}\label{eq-1}
H^{q+1}(X,T^{1,0})\longrightarrow \Hom (H^{n,0}(X), H^{n-1,q+1}(X)):~\xi\longmapsto~i_{\xi}
\end{equation}
is an isomorphism, we recover the following result from \cite[Thm.\,8.4]{Xia19dDol}:
\begin{corollary}\label{coro-}
Let $X$ be a Calabi-Yau manifold, then the deformations of $T_X^{1,0}$-valued $(0,q)$-forms on $X$ are unobstructed.
\end{corollary}
In particular, we see that if $K_X$ is trivial and $\dim H^{n,0}(X)$ is a deformation invariant of $X$ (which is equivalent to the condition that deforming $(n,0)$-forms on $X$ is unobstructed by \cite[Coro.\,1.2]{Xia19dDol}) then the deformations of $T_X^{1,0}$-valued $(0,q)$-forms on $X$ are unobstructed.

\vskip 1\baselineskip \textbf{Acknowledgements.} I would like to thank Prof. Kefeng Liu for his constant encouragement and many useful discussions. I would also like to thank Prof. Bing-Long Chen for his constant support.

\bibliographystyle{alpha}

\begin{thebibliography}{LRY15}

\bibitem[Cat13]{Cat13}
F.~Catanese.
\newblock A superficial working guide to deformations and moduli.
\newblock In {\em Handbook of moduli. Vol. I}, volume~24 of {\em {Advanced
  Lectures in Mathematics}}, pages 161--215. {International Press, Somerville,
  MA}, 2013.

\bibitem[Cle99]{Cle99}
H.~Clemens.
\newblock On the geometry of formal {Kuranishi} theory.
\newblock {\em Advances in Mathematics}, 198(1):311--365, 1999.

\bibitem[LRY15]{LRY15}
K.~Liu, S.~Rao, and X.~Yang.
\newblock Quasi-isometry and deformations of {C}alabi-{Y}au manifolds.
\newblock {\em Inventiones mathematicae}, 199(2):423--453, 2015.

\bibitem[Man04]{Man04}
M.~Manetti.
\newblock Cohomological constraint on deformations of compact {K\"ahler}
  manifolds.
\newblock {\em Advances in Mathematics}, 186(1):125--142, 2004.

\bibitem[Man05]{Ma05}
M.~Manetti.
\newblock Lectures on deformations of complex manifolds.
\newblock arXiv:math/0507286v1 [math.AG], 2005.

\bibitem[Xia18]{Xia18da}
W.~Xia.
\newblock Derivations on almost complex manifolds.
\newblock arXiv:1809.07443v3, 2018.

\bibitem[Xia19]{Xia19dDol}
W.~Xia.
\newblock Deformations of {Dolbeault} cohomology classes.
\newblock arXiv:1909.03592v1 [math.DG], 2019.

\end{thebibliography}

\end{document}